\documentclass[10pt]{amsart}
\usepackage{amsmath,amssymb,latexsym,soul,cite,mathrsfs,amsthm}

\usepackage{color,enumitem,graphicx}
\usepackage[colorlinks=true,urlcolor=blue,
citecolor=red,linkcolor=blue,linktocpage,pdfpagelabels,
bookmarksnumbered,bookmarksopen]{hyperref}
\usepackage[english]{babel}

\usepackage[left=2.9cm,right=2.9cm,top=2.8cm,bottom=2.8cm]{geometry}
\usepackage[hyperpageref]{backref}

\usepackage[colorinlistoftodos]{todonotes}
\makeatletter
\providecommand\@dotsep{5}
\def\listtodoname{List of Todos}
\def\listoftodos{\@starttoc{tdo}\listtodoname}
\makeatother

\numberwithin{equation}{section}
\def\dis{\displaystyle}

\newtheorem{lemma}{Lemma}
\newtheorem{proposition}{Proposition}
\newtheorem{theorem}{Theorem}
\newtheorem*{theorem*}{Theorem}

\newtheorem{remark}{Remark}

\title[A quasilinear Schr\"odinger-Poisson system]
{quasi-linear Schr\"odinger-Poisson system \\
under an exponential critical nonlinearity:\\
existence and asymptotic behaviour of solutions}

\author[G. M. Figueiredo]{Giovany M. Figueiredo}
\author[G. Siciliano]{Gaetano Siciliano}

\address[G. M. Figueiredo]{\newline\indent Universidade de Bras\'ilia-UNB
\newline\indent 
Departamento de Matem\'atica
\newline\indent
CEP: 70910-900, Bras\'ilia, DF, Brazil}
\email{\href{mailto:giovany@unb.br}{giovany@unb.br}}

\address[G. Siciliano]{\newline\indent Departamento de Matem\'atica
\newline\indent 
 Universidade de S\~ao Paulo 
\newline\indent 
Rua do Mat\~ao 1010,  05508-090 S\~ao Paulo, SP, Brazil }
\email{\href{mailto:sicilian@ime.usp.br}{sicilian@ime.usp.br}}

\thanks{The authors are partially supported by
CNPq, Capes,  FAPDF and Fapesp, Brazil.}
\subjclass[2010]{
35Q60,  	
35J10,  	
35J50,  	
35J92,  	
35J61.  	
}
\keywords{Variational methods, nonlocal problems,
Schr\"odinger-Poisson equation, exponential critical growth.}

\begin{document}

\begin{abstract}
In this paper  we consider 
the following quasilinear Schr\"odinger-Poisson system  in 
a bounded domain in $\mathbb{R}^{2}$:
$$
\left\{
\begin{array}[c]{ll}
- \Delta u +\phi u =  f(u) &\ \mbox{in }  \Omega, \\
-\Delta \phi - \varepsilon^{4}\Delta_4 \phi = u^{2} &   \ \mbox{in }  \Omega,\\
u=\phi=0  & \  \mbox{on }  \partial\Omega 
\end{array}
 \right.
$$
depending on the  parameter $\varepsilon>0$. 
The nonlinearity $f$ is assumed to have critical exponencial growth. 
We first prove existence 
of nontrivial solutions $(u_{\varepsilon}, \phi_{\varepsilon})$
and then we show that as $\varepsilon\to0^{+}$ these solutions converges to a 
nontrivial solution of the associated Schr\"odinger-Poisson system, that is by making $\varepsilon=0$
in the system above.

\end{abstract}

\maketitle

\section{Introduction}

In this paper we study the following system
\begin{equation}\label{eq:P}\tag{$P_{\varepsilon}$}
\left\{
\begin{array}[c]{ll}
-\Delta u +\phi u =  f(u) & \ \mbox{in }  \Omega, \medskip\\
 -\Delta \phi -\varepsilon^{4}\Delta_4 \phi =u^{2}& \ \mbox{in }  \Omega,\medskip\\
u=\phi=0 &\ \mbox{on }  \partial\Omega
\end{array}
 \right.
\end{equation}
where $\Omega\subset \mathbb R^{2}$ is a smooth and bounded domain,
$\Delta_{4} = \text{div} (|\nabla \phi|^{2} \nabla \phi)$ is the $4-$Laplacian
 and $f$ satisfies suitable assumptions, allowing to have critical growth.

Problem \eqref{eq:P} is the planar version of the so called 
quasilinear Schr\"odinger-Poisson system which,
after the papers
\cite{BK, IKL} has attracted the attention of mathematicians in these recent
years. However few papers deal with this kind of system.
We cite here  \cite{ILTZ}  where the authors 
consider  the quasilinear Schr\"odinger-Poisson system in the unitary cube in $\mathbb R^{3}$ under periodic 
boundary conditions; they show global existence and uniqueness of solutions.
In \cite{LY} the author proves
 existence and uniqueness of a global mild solution   in the one dimensional case.
In the recent paper \cite{DLMZ}, the problem in $\mathbb R^{3}$ with an 
asymptotically linear $f$ is considered. 
The authors prove existence and the  behaviour of the ground state solutions as  $\varepsilon\to 0^{+}$.
Again the solutions converge to the solution of the ``limit'' problem with $\varepsilon=0$.
Finally in \cite{GG} we studied the problem in $\mathbb R^{3}$ under a critical nonlinearity,
showing again that the solutions converge to a solution of the Schr\"odinger-Poisson system.



\medskip

As explained in \cite{AA} (see also \cite{BK,IKL})
the system appears 
by studying a quantum physical model  of extremely small devices in semi-conductor nanostructures
and takes into account the quantum structure and   the longitudinal field oscillations during the beam propagation.
This is reflected into the fact that 
the  dielectric permittivity depends on the electric field by
 \begin{equation*}\label{eq:dielectr}
 {c}_{\textrm{diel}}(\nabla \phi)= 1+\varepsilon^{4}|\nabla \phi|^{2}, \quad \varepsilon>0 \ \textrm{ and constant.}
 \end{equation*} 
We refere the reader to \cite{GG} where the  system is deduced in the framework of Abelian Gauge Theories.

\medskip

Before to state our results let us introduce some notations. In this paper we fix
an arbitrary  $r>2$ and consider the auxiliary problem
\begin{equation}\label{eq:A}\tag{$A$}
\begin{cases}
-\Delta u = |u|^{r-2}u &   \mbox{ in }  \Omega, \smallskip\\ 
 u=0 & \mbox{ on }\partial\Omega.
\end{cases}
\end{equation}
 Let 
$$
R(u)=\frac 12\displaystyle\int_{\Omega}|\nabla u|^{2}- \frac 1r\displaystyle\int_{\Omega}|u|^{r}
$$
be the  the functional associated to  problem \eqref{eq:A}  and let
$$
\mathcal{N}=\left\{u \in H^{1}_0(\Omega)\setminus\{0\}:  R'(u)[u]=0\right\}.
$$
be the Nehari manifold.
Hereafter $H^{1}_{0}(\Omega)$ is the usual Sobolev space endowed with scalar product and (squared) norm
given by
$$\langle u, v\rangle:=\int_{\Omega}\nabla u\nabla v, \quad \|u\|^{2} =\int_{\Omega}|\nabla u|^{2}.$$
We have $H^{1}_0(\Omega)\hookrightarrow L^{p}(\Omega)$, for $p\geq1$.
The $L^{p}-norm$ will be simply denoted with $|\cdot |_{p}$.

Standard arguments give the existence of a ground state $\mathfrak u\in H^{1}_{0}(\Omega)$ 
for problem \eqref{eq:A} which satisfies
$$
\mathfrak m:=R(\mathfrak u) = \min_{\mathcal{N}} R, \ \ R'(\mathfrak u)=0
$$
and
\begin{eqnarray}\label{estimatetoc_r} 
\mathfrak m=\frac{r-2}{2r}\displaystyle\int_{\Omega}|\mathfrak u|^{r} =\frac{r-2}{2r}\|\mathfrak u\|^{2}.
\end{eqnarray}

\medskip

Now we can state our assumptions on $f$ in order to study problem \eqref{eq:P}.

 Let  $f:\mathbb{R} \rightarrow\mathbb{R}$ be  a continuous
function such that\medskip
\begin{enumerate}[label=(f\arabic*),ref=f\arabic*,start=0]
\item\label{f_{0}} $f(t)=0$ for $t\leq0$, \medskip
\item \label{f_{1}}$ \lim_{t \rightarrow 0}\dis \frac{f(t)}{t}=0,$  \medskip
\item\label{f_{2}}there exists $\alpha_0 >  0$ such that 
$$\displaystyle\lim_{t \to \infty }\frac{f(t)}{\exp(\alpha t^{2})} 
= 
\begin{cases}
0 & \mbox{ for } \alpha > \alpha_{0} \smallskip \\
\infty & \mbox{ for }  \alpha < \alpha_{0}, 
\end{cases}
$$
\medskip

\item\label{f_{3}} there exists $\theta \in(4,+\infty)$ such that
$$
0<\theta F(t)=\theta\int^{t}_{0}f(s)ds \leq tf(t), \quad
\mbox{for all} \ \ t>0,
$$
\item\label{f_{4}} there is  $\tau\geq\tau^{*}(\varepsilon)$ such that 
$$
f(t) \geq \tau t^{r-1}, \quad \forall t\geq0
$$
where 
$$
\tau^*(\varepsilon) := \max\biggl\{
\biggl[\frac{  \theta  \mathfrak m (\alpha_{0}+1) }{\pi(\theta-2)}\biggl]^{(r-2)/2}, 
\biggl[\frac{4\theta \mathfrak m}{(\theta -2) \overline T^{2}(\varepsilon)}\biggl]^{(r-2)/2}
\biggl\},
$$
and where $\overline T(\varepsilon)>0$ will appear later.
\end{enumerate}


\medskip

We would like to highlight that  the model  nonlinearity 
\begin{eqnarray*}\label{exemplodef}
f(t)= 
\begin{cases}
\tau^{*}(\varepsilon) |t|^{r-2}t \exp(\alpha_0 t^{2}) & \mbox{ for } t\geq0\\
0 & \mbox{ for } t\leq0.
\end{cases}
\end{eqnarray*} 
satisfies  all the  assumptions above.

\medskip

We define
$$X:=H^{1}_0(\Omega)\cap W^{1,4}_0(\Omega)$$
which  is a Banach space under the norm
$$\|\phi\|_{X}:=|\nabla \phi|_{2} + |\nabla \phi|_{4}.$$
Note that $X\hookrightarrow L^{\infty}(\Omega)$.

%

\bigskip

By a solution of \eqref{eq:P} we mean a pair 
$(u_{\varepsilon}, \phi_{\varepsilon})\in H^{1}_0(\Omega)\times X$ such that
\begin{eqnarray}\label{eq:ws1}
\forall v\in H^{1}_0(\Omega): \quad \int_{\Omega} \nabla u_{\varepsilon} \nabla v 
+\int_{\Omega} \phi_{\varepsilon} u_{\varepsilon}v
=\int_{\Omega} f(u_{\varepsilon}) v\\
\label{eq:ws2}
 \forall \xi \in X:\quad  \int_{\Omega} \nabla \phi_{\varepsilon} \nabla \xi
+\varepsilon^{4}\int_{\Omega} |\nabla \phi_{\varepsilon}|^{2}\nabla  \phi_{\varepsilon}\nabla \xi 
= \int_{\Omega} \xi u_{\varepsilon}^{2} .
\end{eqnarray}

\medskip

The main results of this paper are the following.
\begin{theorem} \label{teorema1}
Assume that conditions \eqref{f_{0}}-\eqref{f_{4}} hold.
Then,  for every $\varepsilon>0$ problem \eqref{eq:P} admit a solution 
$(u_{\varepsilon}, \phi_{\varepsilon})\in H^{1}_0(\Omega)\times X$.
Moreover  $\phi_{\varepsilon}, u_{\varepsilon}$ are nonnegative.
\end{theorem}

We study also the asymptotic behaviour of the solutions $u_{\varepsilon}, \phi_{\varepsilon}$
as $\varepsilon\to 0^{+}$ obtaining the following

\begin{theorem}\label{th:2}
Let $f:\mathbb R\to \mathbb R $ be a function satisfying
conditions \eqref{f_{0}}-\eqref{f_{2}} and
consider  the  Schr\"odinger-Poisson system \medskip
\begin{equation*}\label{eq:P0}\tag{$P_{0}$}
\left\{
\begin{array}[c]{ll}
-\Delta u +\phi u =  f(u) & \ \mbox{in }  \Omega, \medskip\\
 -\Delta \phi  =u^{2}& \ \mbox{in }  \Omega,\medskip\\
u=\phi=0 &\ \mbox{on }  \partial\Omega.
\end{array}
 \right.
\end{equation*}
If $\{u_{\varepsilon}, \phi_{\varepsilon}\}_{\varepsilon>0}$ are  solutions of \eqref{eq:P}
satisfying also   $\|u_{\varepsilon}\|^{2}\leq 2\pi/( \alpha_{0}+1)$ 
then, \medskip 
\begin{itemize}
\item[1.] $\lim_{\varepsilon\to0^{+}} u_{\varepsilon} = u_{0}$ in $ H^{1}_{0}(\Omega)$,\medskip
\item[2.]  $\lim_{\varepsilon\to0^{+}} \phi_{\varepsilon} = \phi_{0}$ in $H^{1}_{0}(\Omega)$,\medskip
\end{itemize}
where $(u_{0}, \phi_{0})$ is a nontrivial solution of \eqref{eq:P0}.
\end{theorem}

In particular Theorem \ref{th:2} gives the existence of a nontrivial solution
for \eqref{eq:P0}, which we were not able to find in the mathematical literature
under our assumption on $f$.

\medskip

Our contribution in this paper is to give a better understanding on this
quasilinear problem, on which there are just few papers (cited above)
in the literature. Moreover, to the best of our knowledge,
 this is the first paper dealing with the two dimensional
case and involving a critical nonlinearity; and in fact  the main 
difficulties are related to (i) the ``fourth'' order term in the equation
(hence in particular any homogeneity property is lost) and
(ii)  to the critical growth of the nonlinearity. 

\medskip

We find solutions by using variational methods by using Mountain Pass arguments.
Indeed the solutions will be critical points of a functional $J_{\varepsilon}$.
However, to avoid the previous difficulties,
we introduce a suitable truncated functional, $J_{\varepsilon}^{T}$, depending on a parameter $T>0$,
in such a way that  we have compactness at the  Mountain
Pass level of the truncated functional, and even more, we can recover
a critical point  of the untruncated functional.
Then by using suitable estimates with respect to $\varepsilon$ we are able to
show that the solutions of \eqref{eq:P} tends, as $\varepsilon$ tends do zero,
to a nontrivial solution of the Schr\"odiger-Poisson system.

\medskip

The paper is organized as follows.

In Section \ref{sec:VF} we introduce the variational framework, by
defining a  $C^{1}$ functional $J_{\varepsilon}$  naturally associated to \eqref{eq:P}. 

The truncated functional 
$J_{\varepsilon}^{T}$ is 
introduced in Section \ref{sec:trunc}, where we prove also
a suitable estimate on its Mountain Pass level.

In Section \ref{sec:finalsec}  we show that, for a suitable choice of the 
truncation parameter $T$, the Mountain Pass level of the $J_{\varepsilon}^{T}$
 satisfies an estimate which permits  to have compactness and recover a 
critical point of $J_{\varepsilon}$, hence a solution of \eqref{eq:P}, proving Theorem \ref{teorema1}.

Finally in Section \ref{sec:finalissima} we prove Theorem \ref{th:2}.

\medskip

As a matter of notation we use  for brevity the notation $\int_{\Omega} w$ to mean $\int_{\Omega}w(x)dx.$

\section{The variational framework}\label{sec:VF}
Let us start by noticing that
from $\eqref{f_{1}}-\eqref{f_{3}}$, 
for all $\delta >0$ and for all $\alpha>\alpha_0$, there exist constants $C_{\delta },
\widetilde C_{\delta }>0$ such that 
\begin{eqnarray}\label{iff}
\int_\Omega f(u)u \leq \delta  \displaystyle\int_{\Omega}u^2  +C_{\delta}\displaystyle\int_{\Omega}|u|^{q}\exp(\alpha u^{2} )
\end{eqnarray}
and 
\begin{eqnarray}\label{iff1}
\int_\Omega F(u)\leq \delta  \displaystyle\int_{\Omega}u^2  +\widetilde{C}_{\delta}\displaystyle\int_{\Omega}|u|^{q}\exp(\alpha u^{2}), 
\end{eqnarray}
for all $u \in H^{1}_0(\Omega)$ and 
for all $q\geq 0$. 
Let us recall the following Trundiger-Moser inequality.
\begin{proposition}[\cite{Moser}]\label{resolvetudo}
If $\alpha>0$ and $u\in H^{1}_0(\Omega)$, then
$$
\dis\int_{\Omega}\exp\bigl(\alpha u^{2}\bigl)<\infty. 
$$
Moreover if $\alpha<4\pi$  there exists a constant
$C=C(\alpha)>0$ such that  
$$
\sup_{\|u\|\le1}\dis\int_{\Omega}\exp (\alpha u^{2} )\leq C. 
$$
\end{proposition}

It is easy to see that the critical points of the smooth functional
\begin{equation*}\label{eq:F}
\mathcal J_{\varepsilon} (u,\phi) = \frac{1}{2}\|u\|^{2}+\frac{1}{2}\int_{\Omega} \phi u^{2} -\int_{\Omega} F(u) 
-\frac{1}{4}\int_{\Omega} |\nabla \phi|^{2} -\frac{\varepsilon^{4}}{8}\int_{\Omega} |\nabla \phi|^{4} 
\end{equation*}
on $H^{1}_0(\Omega)\times X $
are exactly the weak solutions of \eqref{eq:P}, according to \eqref{eq:ws1} and \eqref{eq:ws2}.
However since this functional $\mathcal J_{\varepsilon}$ is strongly indefinite in the product space
$H^{1}_0(\Omega)\times X $
we use a well known by now  reduction method
which consists in studying a suitable functional of a single variable.
We give the details in the next subsection, which basically consists in solving
for every $u\in H^{1}_{0}(\Omega)$ the second equation of the system and then
substituting in the first equation.

\medskip

The value of $\varepsilon>0$ has to be considered  fixed until the end of 
Section \ref{sec:finalsec}. In Section \ref{sec:finalissima}
where we will pas to the limit  with respect to $\varepsilon$ to prove Theorem \ref{th:2}.

\subsection{Study of the quasilinear Schr\"odinger-Poisson equation}

Let us study here the second equation of the system \eqref{eq:P}.
Note that, 
for  every $u\in H^{1}_{0}(\Omega)$ the map (with abuse of notations)
\begin{equation*}\label{eq:}
u^{2}: \phi \in X \longmapsto \int _{\Omega}\phi u^{2}\in \mathbb R
\end{equation*}
is linear and continuous, hence $u^{2}\in X'$. Then the unique solution of
\begin{equation}\label{eq:2eq}
\left\{
\begin{array}[c]{ll}
-\Delta \phi -\varepsilon^{4}\Delta_{4}\phi = u^{2}& \ \mbox{in }  \Omega, \medskip\\
\phi =0 & \ \mbox{in }  \Omega,\\
\end{array}
 \right.
\end{equation} is the unique critical point (the minimum) of the functional
$$\phi\in X\longmapsto \frac{1}{2}\int_{\Omega} |\nabla \phi|^{2} +\frac{\varepsilon^{4}}{4}\int_{\Omega} |\nabla \phi|^{4}-\int_{
\Omega} \phi u^{2}\in \mathbb R.$$
Hence it makes sense to consider the map
\begin{equation}\label{eq:phi}
\Phi_{\varepsilon}: u\in H^{1}_{0}(\Omega) \longmapsto \phi_{\varepsilon}(u) \in X
\end{equation}
where $\phi_{\varepsilon}(u)$ is the unique solution of \eqref{eq:2eq}.
The continuity of $\Phi_{\varepsilon}$ follows by the next result, whose proof is exactly as in
\cite[Lemma 1]{GG}.
\begin{lemma}\label{solutionoperator}
Let $g_{n}\to g$ in $X'$. Then, 
we have
$$\int_{\Omega} |\nabla \phi_{\varepsilon}(g_{n})|^{2} \to \int_{\Omega} |\nabla \phi_{\varepsilon}(g)|^{2}, \quad
 \int_{\Omega} |\nabla \phi_{\varepsilon}(g_{n})|^{4}\to \int_{\Omega} |\nabla \phi_{\varepsilon}(g)|^{4}.$$
In particular the operator $\Phi_{\varepsilon}$ in \eqref{eq:phi} is continuous and
$$\phi_{\varepsilon}(g_{n})\to \phi_{\varepsilon}(g) \text{ in } \  L^{\infty}(\Omega).$$

\end{lemma}

\medskip

 In the remaining of the paper, $\phi_{\varepsilon}(u)$ will always denote the unique solution of \eqref{eq:2eq}
 with fixed $u$. Note that it
satisfies
\begin{equation}\label{eq:sostituicao}
\int_{\Omega} |\nabla \phi_{\varepsilon}(u)|^{2} + \varepsilon^{4}\int_{\Omega}| \nabla\phi_{\varepsilon}(u)|^{4}
 = \int_{\Omega} \phi_{\varepsilon}(u) u^{2}.
\end{equation}
The next result will be usefull in the following. 
\begin{lemma}\label{lem:facile}
Let $q\in[1,+\infty)$. If $\{u_{n}\}$ converges to some $w$ in $L^{q}(\Omega)$ 
 then, \medskip
 \begin{itemize}
 \item[(a)] $\displaystyle \lim_{n\to+\infty}\int_{\Omega} |\nabla \phi_{\varepsilon}(u_{n})|^{2} = \int_{\Omega} |\nabla \phi_{\varepsilon}(w)|^{2}$, \medskip
\item[(b)] $ \displaystyle\lim_{n\to+\infty}\int_{\Omega} |\nabla \phi_{\varepsilon}(u_{n})|^{4} =\int_{\Omega} |\nabla \phi_{\varepsilon}(w)|^{4}$, \medskip
\item[(c)]$ \displaystyle\lim_{n\to+\infty}\int_{\Omega}\phi_{\varepsilon}(u_{n})u_{n}^{2} = \int_{\Omega}\phi_{\varepsilon}(w) w^{2} $, \medskip
\item[(d)]$ \displaystyle\lim_{n\to+\infty}\phi_{\varepsilon}(u_{n}) =\phi_{\varepsilon}(w) $ in $L^{\infty}(\Omega)$,\medskip
\item[(e)] for all $v\in H^{1}_{0}(\Omega): \displaystyle\lim_{n\to+\infty}\int_{\Omega}\phi_{\varepsilon}(u_{n})u_{n}v = \int_{\Omega}\phi_{\varepsilon}(w) wv $.
 \end{itemize}
\end{lemma}

\begin{proof}
Under our assumptions we have,
$$\|u_{n}^{2} - w^{2}\| = \sup_{\|\phi\|_{X}=1} \Big| \int_{\Omega} \phi (u_{n}^{2} - w^{2})\Big|\leq |\phi|_{q'} |u_{n}^{2} - w^{2}|_{q}\leq C |u_{n}^{2} - w^{2}|_{q}\to 0.$$
Then we can apply Lemma \ref{solutionoperator} and \eqref{eq:sostituicao} and  conclude the proof
of $(a),(b),(c),(d)$. The proof of $(e)$ follows by using an H\"older inequality and $(d)$.
\end{proof}

Let $G(\Phi_{\varepsilon})$ be the graph of the map $\Phi_{\varepsilon}: u\in H^{1}_0(\Omega)\mapsto \phi_{\varepsilon}(u)\in X$.

Since the functional  $\mathcal J_{\varepsilon}$ is $C^{2}$, classical arguments using  
the Implicit Function Theorem (see e.g. \cite{BF} for  the Schr\"odinger-Poisson system)
give that
$$G(\Phi_{\varepsilon})=\left\{(u,\phi)\in H^{1}_0(\Omega)\times X: \partial_{\phi}\mathcal J_{\varepsilon}(u,\phi) = 0\right\}.$$
and actually  $\Phi_{\varepsilon}\in C^{1}\Big(H^{1}_0(\Omega); X \Big).$

As a consequence, the functional (recall \eqref{eq:sostituicao})
$$
J_{\varepsilon}(u):=\mathcal J_{\varepsilon}(u,\Phi_{\varepsilon}(u)) =
\frac 12 \|u\|^{2}+ \frac 14 \displaystyle\int_{\Omega}|\nabla \phi_{\varepsilon}(u)|^{2}  + \frac{3\varepsilon^{4}}{8} \displaystyle\int_{\Omega}|\nabla \phi_{\varepsilon}(u)|^{4}  
-\displaystyle\int_{\Omega}F(u) 
$$
is of class $C^{1}$ and in particular we have
\begin{eqnarray*}\label{eq:}
J_{\varepsilon}'(u)[v] &= &\partial_{u} \mathcal J_{\varepsilon}(u,\phi_{\varepsilon}(u))[v]+
\partial_{\phi}\mathcal J_{\varepsilon}(u,\phi_{\varepsilon}(u))\circ \Phi_{\varepsilon}'(u) [v] \\ 
&=&\partial_{u} \mathcal J_{\varepsilon}(u,\phi_{\varepsilon}(u))[v].
\end{eqnarray*}
Then
$$J_{\varepsilon}'(u)[v]  = \int_{\Omega} \nabla u \nabla v+\int_{\Omega}\phi_{\varepsilon}(u)uv
-\int_{\Omega} f(u)v $$
which shows that finding a critical point $u_{\varepsilon}$ of $J_{\varepsilon}$ is equivalent to obtain a solution
$(u_{\varepsilon}, \phi_{\varepsilon})$ of \eqref{eq:P} where $\phi_{\varepsilon}=\Phi_{\varepsilon}(u_{\varepsilon})$.
We are then reduced to study the problem
\begin{equation*}
\begin{cases}
- \Delta u +\phi_{\varepsilon}(u) u =  f(u) &  \mbox{in }  \Omega\\
 u=0 & \mbox{on }  \partial\Omega.
\end{cases}
\end{equation*}

For brevity we introduce the functional
\begin{equation*}\label{eq:I}
I_{\varepsilon}:u\in H^{1}_0(\Omega)\longmapsto \dis\frac{1}{4}\int_{\Omega}|\nabla \phi_{\varepsilon}(u)|^{2} +  
\frac{3\varepsilon^{4}}{8}\int_{\Omega} |\nabla \phi_{\varepsilon}(u)|^4
\in \mathbb R
\end{equation*}
so that we can write
$$J_{\varepsilon}(u) = \frac{1}{2}\|u\|^{2} + I_{\varepsilon}(u) -\int_{\Omega} F(u)
.$$

\section{The truncated functional}\label{sec:trunc}

In order to overcome  the ``growth'' of order $4$ in $I_{\varepsilon}$, let us define a truncation for the functional $J_{\varepsilon}$
 in the following way. Consider a smooth cut-off  
 and non-increasing function $\psi:[0,+\infty)\to \mathbb [0,+\infty)$ such that
$$
\left\{\begin{array}{lll} 
\psi(t) = 1, & t\in[0,1],\smallskip\\
0 \leq \psi(t) \leq 1, & t\in(1,2),\smallskip\\
\psi(t) = 0, &t\in[2,\infty),\smallskip\\
|\psi'|_{\infty}\leq 2.
\end{array}\right.
$$
We define  $h_T(u) := \psi\left({\|u\|}^2/{T^2}\right)$ and the  truncated functional 
$J_{\varepsilon}^{T}:H^{1}_0(\Omega)\rightarrow\mathbb{R}$ given by 
\begin{eqnarray*}
J_{\varepsilon}^{T}(u) &:= &\frac 12 \|u\|^{2}+ 
h_{T}(u)\biggl[\frac 14 \int_{\Omega}|\nabla \phi_{\varepsilon}(u)|^{2}  
+ \frac{3\varepsilon^{4}}{8} \int_{\Omega}|\nabla \phi_{\varepsilon}(u)|^{4}
 \biggl] 
- \displaystyle\int_{\Omega}F(u)\\
&=&\frac 12 \|u\|^{2}+  h_{T}(u)I_{\varepsilon}(u)- \displaystyle\int_{\Omega}F(u) 
.
\end{eqnarray*}
The functional $J_{\varepsilon}^{T}$ is $C^{1}$ with differential given, for all $u,v \in H^{1}_0(\Omega)$, by
\begin{equation}\label{eq:derivada}
(J_{\varepsilon}^{T})'(u)[v]=\langle u, v\rangle
+\frac{2}{T^{2}}\psi'\left(\frac{\|u\|^{2}}{T^{2}}\right)\langle u, v\rangle I_{\varepsilon}(u)
 +h_{T}(u)\int_{\Omega} \phi_{\varepsilon}(u) u v 
-\int_{\Omega} f(u)v.
\end{equation}


\subsection{The Mountain Pass Geometry for $J_{\varepsilon}^{T}$}

The next two results deal with the Mountain Pass geometry for the functional $J^{T}_{\varepsilon}$
where $\varepsilon, T>0$.

 We point out  that the Mountain Pass structure
of $J_{\varepsilon}^{T}$ does not depend on $\varepsilon$.
In other words, 
\begin{itemize}
\item $\beta, \rho$ in Lemma \ref{geometria1} 
does not depend on $\varepsilon$, neither on $T.$
\item $e_{T}$ in Lemma \ref{segundageometria1}
just depend on $T$, and not on $\varepsilon$.
\end{itemize}
The reason of this independence on $\varepsilon$  is because the terms 
involving $\varepsilon$ (that is $I_{\varepsilon}$) is suitably thrown away.

\begin{lemma}\label{geometria1}
Assume that conditions  \eqref{f_{1}} and \eqref{f_{2}} hold.
Then,  there exists
 numbers $\rho, \beta>0 $ such that,
$$
\forall T>0, \quad 
J_{\varepsilon}^{T}(u)\geq \beta,\quad  \text{whenever }\
\|u\|=\rho.
$$
\end{lemma}
\begin{proof}
Let  $\alpha > \alpha_0$ and use  \eqref{iff1} with $q>2$: 
 taking $\delta>0$ sufficiently small there exists $D_1>0$ such that
\begin{eqnarray*}
J_{\varepsilon}^{T}(u)\geq  D_1\|u\|^2 - \widetilde{C}_\delta\displaystyle\int_{\Omega}|u|^{q}\exp( \alpha  u^{2}). 
\end{eqnarray*}
Using H\"older's inequality 
\begin{eqnarray*}
J_{\varepsilon}^{T}(u)&\geq&  D_1\|u\|^2 - \widetilde{C}_\delta\biggl(\displaystyle\int_{\Omega} |u|^{2q }\biggl)^{1/2}
\biggl(\int_{\Omega} \exp\left( 2\alpha  \|  u\|^{2}\frac{u^{2}}{\|u\|^{2}}\right ) \biggl)^{1/2}\\
&\geq &D_1\|u\|^2 - D_{2}\|u\|^{q}\biggl(\int_{\Omega} \exp\left( 2\alpha  \|  u\|^{2}\frac{u^{2}}{\|u\|^{2}}\right ) \biggl)^{1/2}.
\end{eqnarray*}
Then we can choose $\rho_1=\|u\|>0$  small enough
such that $2\alpha \rho^{2}_{1}< 4\pi$,  so that, by Proposition \ref{resolvetudo} we get 
\begin{eqnarray*}
J_{\varepsilon}^{T}(u)\geq D_1  \rho^{2}_{1} - D_2 \rho^{q}_{1},
\end{eqnarray*}
for some $D_2>0$. Thus there exists $\beta>0$ such that $J^{T}(u)\geq \beta >0$, for all $0<\rho<\rho_1$ which proves the Lemma.
\end{proof}

\begin{lemma}\label{segundageometria1}
Assume that conditions \eqref{f_{4}} hold. 
Then for every $T>0$, there exists $e_{T} \in H^{1}_{0}(\Omega)$ such that
$$ \quad J_{\varepsilon}^{T}(e_{T})<0\quad \text{and } \quad \| e_{T}\| >\rho,$$
where $\rho$ is given in Lemma \ref{geometria1}.
\end{lemma}
\begin{proof}
Let $T>0$ be fixed. Let now $v\in C^{\infty}_{0}(\Omega)$, positive, with $\|v\|=1$.
Using \eqref{f_{4}} and considering $t>2T$,  we get
$$
J_{\varepsilon}^{T}(tv)\leq \frac 12
t^{2}-\tau
\frac{t^{r}}{r}\int_{\Omega} v^{r}$$
Since $2< r$,  the result follows by choosing some 
$t_{*}>2T$ large enough and setting
$e_{T}:=t_{*}v$.
\end{proof}

\medskip

%

\section{Proof of Theorem \ref{teorema1}}\label{sec:finalsec}

Since for every $\varepsilon,T>0$ 
 the functional $J_{\varepsilon}^{T}$ satisfies the geometric assumptions of 
Mountain Pass Theorem  (see \cite{Ambrosetti}), we know that 
there exists a $(PS)$ sequence at this level, that is a sequence $\{u_{n}\}\subset
H^{1}_0(\Omega)$  satisfying
$$
J_{\varepsilon}^{T}(u_{n})\rightarrow c_{\varepsilon}^{T}>0 \ \ \mbox{and} \ \
(J_{\varepsilon}^{T})'(u_{n})\rightarrow 0,
$$
where
$$
c_{\varepsilon}^{T} := \dis\inf_{\gamma \in \Gamma_{\varepsilon}^{T}} \dis\max_{t \in [0,1]}
J_{\varepsilon}^{T}(\gamma(t))>0
$$
and
$$
\Gamma_{\varepsilon}^{T} := \left\{ \gamma \in C\left([0,1],H^{1}_0(\Omega)\right) : \gamma(0)=0,
\ J_{\varepsilon}^{T}(\gamma(1)) < 0\right\}.
$$
It is clear that this sequence $\{u_{n}\}$ should depend also on $\varepsilon$ and $T$
but we omit this for simplicity.

Observe that there exists $k>0$ such that $0< k\leq c_{\varepsilon}^{T}$
for all $\varepsilon, T$, by Lemma \ref{geometria1}.
Moreover  since $e_{T}$ found in Lemma \ref{segundageometria1}
 does not depends on  $\varepsilon$,  by setting
\begin{equation*}
\gamma_{*}: t\in[0,1]\mapsto t e_{T}\in H_{0}^{1}(\Omega)
\end{equation*}
we get 
$\gamma_{*}\in\bigcap_{\varepsilon>0}\Gamma_{\varepsilon}^{T}$.

\medskip

Our next aim is to show that for a suitable choice of $T>0$
(see Lemma \ref{estimateonc}) the $(PS)$  sequence $\{u_{n}\}$ given above
for $J_{\varepsilon}^{T}$ at level $c_{\varepsilon}^{T}$
is bounded   and is actually a  $(PS)$ sequence for the untruncated functional $J_{\varepsilon}$
(see Lemma \ref{estimatetominimizingsequence}).

First few preliminaries are in order.
It is well-known that, for every $\varepsilon,T>0$
 there is a unique $\mathfrak t_{\varepsilon,T}>0$ such that $c_{\varepsilon}^{T}\leq J_{\varepsilon}^{T}( \mathfrak t_{\varepsilon,T} \mathfrak u)$.
 Recall that $\mathfrak u$ is the ground state of the auxiliary problem \eqref{eq:A}.

The important fact now is that there is a  bound on $\mathfrak t_{\varepsilon,T}$ independent on $T$.
\begin{lemma}\label{parametrolimitado}
There exists $K_{\varepsilon}>0$, such that for every $T>0$
$$
\mathfrak t_{\varepsilon,T}\leq K_{\varepsilon}.
$$
\end{lemma}
\begin{proof}
Since $(J^{T}_{\varepsilon})'(\mathfrak t_{\varepsilon,T} \mathfrak u )[\mathfrak t_{\varepsilon,T}\mathfrak u] = 0$ and $h'(\mathfrak t_{\varepsilon,T}\mathfrak u) \leq 0$, by \eqref{eq:derivada} we easily get
$$\int_{\Omega} f(\mathfrak t_{\varepsilon,T}\mathfrak u )\mathfrak t_{\varepsilon,T}\mathfrak u \leq \mathfrak t_{\varepsilon,T}^{2}\|\mathfrak u\|^{2} + h_{T}(\mathfrak t_{\varepsilon,T}\mathfrak u) I'_{\varepsilon}(\mathfrak t_{\varepsilon,T}\mathfrak u)$$
and then, by hypothesis \eqref{f_{4}}, 
$$\mathfrak t_{\varepsilon,T}^{r} \|\mathfrak u\|^{2} \leq  \tau \mathfrak t_{\varepsilon,T}^{r}\int_{\Omega} |\mathfrak u|^{r} \leq \int_{\Omega} 
f(\mathfrak t_{\varepsilon,T} \mathfrak u) \mathfrak t_{\varepsilon,T} \mathfrak u \leq \mathfrak t_{\varepsilon,T}^{2}\|\mathfrak u\|^{2} + h_{T}(\mathfrak t_{\varepsilon,T}\mathfrak u) I'_{\varepsilon}(\mathfrak t_{\varepsilon,T}\mathfrak u).
$$It follows that, if $\lim_{T\to +\infty} \mathfrak t_{\varepsilon,T} = +\infty$, then 
for $T$ larger and larger $\mathfrak t_{\varepsilon,T}^{r} \|\mathfrak u\|^{2} \leq \mathfrak t_{\varepsilon,T}^{2} \|\mathfrak u\|^{2}$
which is not possible, being $r>2$.
\end{proof}

Observe that all we have done up to now is true for every $T>0$.
Now we will choose a particular value of $T$.

\begin{lemma}\label{estimateonc}
Let $K_{\varepsilon}$ be the value given in Lemma \ref{parametrolimitado}.
For $\overline T(\varepsilon):=K_{\varepsilon}\| \mathfrak u\|/{2},$ the 
Mountain Pass value  $c^{\overline T(\varepsilon)}_{\varepsilon}$ satisfies
$$
0<c_{\varepsilon}^{\overline T(\varepsilon)} \leq  \frac{\pi (\theta-2)}{\theta   (\alpha_{0}+1)}.
$$
\end{lemma}
\begin{proof}
 Using  \eqref{f_{4}}
 and once that $h_{\overline T}(K_{\varepsilon} \mathfrak u)=\psi \left( K_{\varepsilon}^{2}\|\mathfrak u\|^{2} / \overline T^{2}(\varepsilon)\right)=0$, we obtain
$$
c_{\varepsilon}^{\overline T(\varepsilon)}\leq J_{\varepsilon}^{\overline T(\varepsilon)}(  \mathfrak t_{\varepsilon,\overline T(\varepsilon)} \mathfrak u) = \frac{\mathfrak t_{\varepsilon,\overline T(\varepsilon)}^{2}}{2}
\|\mathfrak u\|^{2} -\tau \frac{\mathfrak t_{\varepsilon,\overline T(\varepsilon)}^{r}}{r} \int_{\Omega}| \mathfrak u|^{r} = \biggl(\frac{\mathfrak t_{\varepsilon,\overline T(\varepsilon)}^{2}}{2}- \tau\frac{\mathfrak t_{\varepsilon,\overline T(\varepsilon)}^{r}}{r}\biggl)\int_{\Omega}| \mathfrak u|^{r} .
$$
Using (\ref{estimatetoc_r}), we have
$$
c_{\varepsilon}^{\overline T(\varepsilon)}\leq  \frac{2r \mathfrak m }{r-2}\max_{\xi\geq 0}\left\{ \frac{\xi^{2}}{2}- \tau\frac{\xi^{r}}{r}\right\}
=\frac{\mathfrak m }{ \tau^{2/(r-2)}}\leq  \frac{\pi (\theta-2)}{\theta   (\alpha_{0}+1)}
$$
finishing the proof.
\end{proof}

\begin{remark}
It is worth to point out that the bound on $c^{\overline T(\varepsilon)}_{\varepsilon}$ does not depend on $\varepsilon$.
\end{remark}

From now on we will consider the truncated functional with the value of $\overline T(\varepsilon)$ given in the above Lemma. The reason is explained in the next 

\begin{lemma}\label{estimatetominimizingsequence}
Let $\varepsilon>0$ be fixed and let $\{u_{n}\}\subset H^{1}_0(\Omega)$ 
be the $(PS)$ sequence     
for the functional $J_{\varepsilon}^{\overline T(\varepsilon)}$ at 
level $c_{\varepsilon}^{\overline T(\varepsilon)}$ given above. Then,
$$
\limsup_{n \to \infty}\|u_n\|^{2}\leq   \min\left\{ \frac{2\pi}{\alpha_{0}+1 }, \frac{\overline{T}^{2}(\varepsilon)}{2}
\right\}.
$$
As a consequence $\|u_{n}\|< \overline T(\varepsilon)$ and then
$\{u_{n}\}$ is also a $(PS)$ sequence for the untruncated functional
$J_{\varepsilon}$ at level $c_{\varepsilon}^{\overline T
(\varepsilon)}>0$.
\end{lemma}
\begin{proof}
Using the fact 
that $\theta>4$ we get
\begin{eqnarray*}\label{eq:}
c_{\varepsilon}^{\overline T(\varepsilon)}& =&J_{\varepsilon}^{\overline T(\varepsilon)}(u_n)-\frac{1}{\theta}(J_{\varepsilon}^{\overline T(\varepsilon)})'(u_n)[u_n] +o_n(1)\\
&\geq& \frac{\theta -2}{2\theta}\|u_n\|^{2} +o_n(1).
\end{eqnarray*}
By  Lemma \ref{estimateonc} we deduce
$$
\|u_{n}\|^{2} \leq \frac{2\theta }{\theta - 2}\frac{\mathfrak m }{ \tau^{2/(r-2)}}
 +o_{n}(1).
$$
Since $\tau \geq \tau^{*}(\varepsilon)$ in \eqref{f_{4}},  then few computations show that
$$
\|u_{n}\|^{2} < \min\left\{ \frac{2\pi}{\alpha_{0}+1}, \frac{\overline T^{2}(\varepsilon)}{2}\right\}+o_n(1)
$$
and the first part holds. Since
$\|u_{n}\|\leq \overline T(\varepsilon)/\sqrt{2}< \overline T(\varepsilon)$,
the Lemma is completely proved.
\end{proof}

In view of the previous Lemma, there exists $u_{\varepsilon}\in H^{1}_{0}(\Omega)$
such that
$u_{n}\rightharpoonup u_{\varepsilon}$ in $H^{1}_{0}(\Omega)$.
In particular we have a bound on $\| u_{\varepsilon}\|$ independent of
$\varepsilon$; this fact will be used in Section \ref{sec:finalissima}.

The next result deal with the convergence of the nonlinear term $f$.
\begin{lemma}
\label{convergenciafeFcritical}
The $(PS)$ sequence  $\{u_{n}\}$  for the functional $J_{\varepsilon}$ at the
level $c_{\varepsilon}^{ \overline T(\varepsilon)}$ is such that
$$
\int_{\Omega}f(u_{n})u_{n}\rightarrow \displaystyle\int_{\Omega}f(u_{\varepsilon})u_{\varepsilon}, \quad
\int_{\Omega}f(u_{n})u_{\varepsilon}\rightarrow \displaystyle\int_{\Omega}f(u_{\varepsilon})u_{\varepsilon}.
$$
\end{lemma}
\begin{proof}
Let us prove just the first limit since the second one is similar.

We can assume that
\begin{eqnarray*}\label{eq:}
&u_{n}\rightharpoonup u_{\varepsilon} \ \ \mbox{in} \ \ H^{1}_0(\Omega),\\
&u_{n}\rightarrow u_{\varepsilon} \ \ \mbox{in}\ \ L^{p}(\Omega), \ \ p\ge1, \\
&u_{n}(x)\rightarrow u_{\varepsilon}(x) \ \ \mbox{ a.e.} \ \ \mbox{in} \ \ \Omega
\end{eqnarray*}
and
$f(u_{n}(x))u_{n}(x)\rightarrow f(u_{\varepsilon}(x))u_{\varepsilon}(x)$ a.e. in $ \Omega$.
%
%
%
We apply 
 inequality (\ref{iff}) with $q=1$ and $\alpha=\alpha_{0}+1$:  for  $\delta>0$, there exists $C_{\delta}>0$ such that
\begin{equation}\label{eq:stimafn}
f(u_{n}(x))u_{n}(x)\leq \delta u^{2}_{n}(x) +C_{\delta}|u_{n}(x)|\exp \left( (\alpha_{0}+1) u^{2}_{n}(x)\right).
\end{equation}

If we  show that
\begin{equation}\label{eq:h}
\exists h\in L^{1}(\Omega): \ |f(u_{n}(x)) u_{n}(x)| \leq h(x) \quad \text{ a.e. } x\in \Omega,
\end{equation}
then by the Dominated Convergece Theorem we obtain the first limit in the Lemma.
So let us estimate both terms in the right hand side of \eqref{eq:stimafn}.

Clearly $\{u^{2}_{n}\}$ converges in $L^{1}(\Omega)$, then up to subsequences,
\begin{equation}\label{eq:h1}
\exists h_{1}\in L^{1}(\Omega): \  u^{2}_{n}(x)\leq h_{1}(x) \quad \text{ a.e. } x\in \Omega.
\end{equation}

Let us bound now $g_{n}:=|u_{n}|\exp \left( (\alpha_{0}+1)u^{2}_{n}\right)$
by some $h_{2}\in L^{1}(\Omega)$.
Of course 
\begin{equation}\label{eq:BL1}
g_{n}(x)\to |u_{\varepsilon}(x)| \exp( (\alpha_{0}+1) u_{\varepsilon}^{2}(x))\ \ \text{ a.e.  } x\in \Omega.
\end{equation}
Now by  Lemma \ref{estimatetominimizingsequence}, 
by choosing $p\in (1,2)$  
 we have
\begin{eqnarray*}
\displaystyle\limsup_{n \to \infty}\|u_n\|^{2}\leq \frac{2\pi}{\alpha_{0}+1 } < \frac{4\pi}{p(\alpha_{0}+1) } 
\end{eqnarray*}
  and then we conclude, by  
Proposition \ref{resolvetudo}, that
\begin{eqnarray*}
\int_{\Omega}\exp \left(p(\alpha_{0}+1)  u^{2}_{n}\right) 
=  \int_{\Omega}\exp \biggl(p( \alpha_{0}+1) \|u_n\|^{2}
\frac{u^{2}_{n}}{\|u_n\|^{2}}\biggl)\leq 
  C
\end{eqnarray*}
where $C$ does not depend on $n$.
Since $\exp \left((\alpha_{0}+1) u^{2}_{n}(x) \right) \to \exp \left((\alpha_{0}+1) u_{\varepsilon}^{2}(x)\right)
$ a.e. in $ \Omega$ 
we infer
\begin{eqnarray}\label{6}
\exp \left(( \alpha_{0}+1) u_{n}^{2}\right)\rightharpoonup \exp \left( (\alpha_{0}+1) u_{\varepsilon}^{2}\right) \ \ \mbox{in} \ \ L^{p}(\Omega),
\end{eqnarray}
see e.g. and \cite[Lemma 4.8]{Kavian}.
Of course it is also
\begin{eqnarray}\label{3}
|u_{n}|\ \to |u_{\varepsilon}| \ \ \mbox{in} \ \ L^{p'}(\Omega), \ \text{ where } p^{-1}+p'^{-1} =1
\end{eqnarray}
and then by \eqref{6} and \eqref{3}
\begin{equation*}
\int_{\Omega} g_{n} =\int_{\Omega}|u_{n}|\exp\left({(\alpha_{0}+1) u_{n}^{2}}\right) \to \int_{\Omega}|u_{\varepsilon}|\exp\left ({(\alpha_{0}+1) u_{\varepsilon}^{2}}\right).
\end{equation*}
But then by \eqref{eq:BL1},
we can invoke the the Brezis-Lieb Lemma
and  deduce that $g_{n}\to|u_{\varepsilon}(x)| \exp\left( (\alpha_{0}+1) u_{\varepsilon}^{2}\right) $ in $L^{1}(\Omega)$,
so that 
(possibly passing to a subsequence)
\begin{eqnarray}\label{eq:h2}
\exists h_{2}\in L^{1}(\Omega): \ g_{n}(x) = |u_{n}(x)| \exp\left((\alpha_{0}+1) u_{n}^{2}(x)\right) \leq h_{2}(x).
\end{eqnarray}
Then by \eqref{eq:h1} and \eqref{eq:h2} we deduce  \eqref{eq:h}.
\end{proof}

%
%

\medskip

Now we can conclude the proof of Theroem \ref{teorema1}.

Since $J_{\varepsilon}'(u_{n})[u_{n}]=o_n(1)$ and $J_{\varepsilon}'(u_{n})[u_{\varepsilon}]=o_n(1)$,
we have
$$o_{n}(1) = \|u_{n}\|^{2} + \int_{\Omega} \phi_{\varepsilon}(u_{n})u_{n}^{2}  - \int_{\Omega}f(u_{n})u_{n}
-\langle u_{n}, u_{\varepsilon}\rangle - \int_{\Omega} \phi_{\varepsilon}(u_{n})u_{n} u_{\varepsilon}
+\int_{\Omega} f(u_{n})u_{\varepsilon}.$$
Then by  the fact that $u_{n}\rightharpoonup u_{\varepsilon}$ in $H^{1}_{0}(\Omega)$,
by Lemma \ref{lem:facile} items $(c),(e)$, with $w:=u_{\varepsilon}$, 
and  Lemma \ref{convergenciafeFcritical}
 we conclude that
$$
\|u_n\| \to \|u_{\varepsilon}\|
$$
and then $u_{n}\to u_{\varepsilon}$ in $H^{1}_{0}(\Omega)$. 

Then  we deduce that $u_{\varepsilon}$ is a critical point of $J_{\varepsilon}$ 
at level $c_{\varepsilon}^{\overline T(\varepsilon)}$ and then 
setting $\phi_{\varepsilon}:= \phi_{\varepsilon}(u_{\varepsilon})$, we have that
$(u_{\varepsilon},  \phi_{\varepsilon})$ is a solution of \eqref{eq:P}

Moreover is easy to see that $\phi_{\varepsilon}(u_{\varepsilon})\geq0$: this is achieved 
by multiplying the second equation in \eqref{eq:P} by $\phi_{\varepsilon}(u_{\varepsilon})^{-}$,
its negative part,
and integrating. Then arguing similarly for
the equation 
$$-\Delta u_{\varepsilon} +\phi_{\varepsilon}(u_{\varepsilon}) = f(u_{\varepsilon})$$
we see that $u_{\varepsilon}\geq0$ and the proof of Theorem
\ref{teorema1} is concluded.

\section{Proof of Theorem \ref{th:2}}\label{sec:finalissima}
All the limits in this section are taken as $\varepsilon\to 0^{+}$.
We denote also with $o_{\varepsilon}(1)$ a quantity which tends to zero as $\varepsilon\to 0^{+}.$



Hence let $\{u_{\varepsilon},\phi_{\varepsilon}\}_{\varepsilon>0}$ be solutions of 
\begin{equation*}
\left\{
\begin{array}[c]{ll}
-\Delta u +\phi u =  f(u) & \ \mbox{in }  \Omega, \medskip\\
 -\Delta \phi -\varepsilon^{4}\Delta_4 \phi =u^{2}& \ \mbox{in }  \Omega,\medskip\\
u=\phi=0 &\ \mbox{on }  \partial\Omega
\end{array}
 \right.
\end{equation*}
where $f$ satisfies just \eqref{f_{0}}-\eqref{f_{2}}. 
We know that $u_{\varepsilon}$ is a critical point of 
$$J_{\varepsilon}(u) = \frac{1}{2}\|u\|^{2} + I_{\varepsilon}(u) -\int_{\Omega} F(u)
 $$ 
 and by assumptions 
$\|u_{\varepsilon}\|^{2}<2\pi/(\alpha_{0}+1)$.
Then there exists $u_{0}$ such that $u_{\varepsilon}\rightharpoonup u_{0}$
in $H^{1}_{0}(\Omega)$
as $\varepsilon\to 0^{+}.$
Let us show this convergence is strong. 
Denote with $\phi_{0}(u_{0})\in H^{1}_{0}(\Omega)$ the unique  solution of
 $$
 \begin{cases}
-\Delta \phi = u_{0}^{2} &  \text{ in } \ \Omega, \\
\phi= 0 & \text{ on } \partial\Omega.
 \end{cases}
$$
We need now the following
\begin{lemma}\label{lem:kavianlem}
 It holds
$$
\lim_{\varepsilon\to0^{+}}\phi_{\varepsilon}(u_{\varepsilon}) = \phi_{0}(u_{0}) \    \text{ in  } \ H^{1}_{0}(\Omega)
\quad 
\text{ and } \quad 
\lim_{\varepsilon\to 0^{+}}\varepsilon \phi_{\varepsilon}(u_{\varepsilon}) = 0\ \text{ in } \  W^{1,4}_{0}(\Omega).
$$
\end{lemma}
\begin{proof}
It is done exactly as in \cite[Lemma 3.2]{BK}), observing that we have convergence
$u_{\varepsilon}\to u_{0}$ in $L^{p}(\Omega)$ for $p\geq1$.
\end{proof}

As at the end of the previous Section,  by combining the identities
 $J'_{\varepsilon}(u_{\varepsilon})[u_{\varepsilon}] = 0$ and  $J'_{\varepsilon}(u_{\varepsilon})[u_{0}]=0$
 we deduce 
 \begin{equation}\label{eq:eps}
 0 = \|u_{\varepsilon}\|^{2} + \int_{\Omega} \phi_{\varepsilon}(u_{\varepsilon})u_{\varepsilon}^{2}  - \int_{\Omega}f(u_{\varepsilon})u_{\varepsilon}
-\langle u_{\varepsilon}, u_{0}\rangle - \int_{\Omega} \phi_{\varepsilon}(u_{\varepsilon})u_{\varepsilon} u_{0} +\int_{\Omega} f(u_{\varepsilon})u_{0}.
 \end{equation}
 Observe that
 $$\left| \int_{\Omega}\phi_{\varepsilon}(u_{\varepsilon})u_{\varepsilon}^{2} - \int_{\Omega}\phi_{\varepsilon}(u_{\varepsilon})u_{\varepsilon}u_{0}  \right| \leq |\phi_{\varepsilon}(u_{\varepsilon})|_{3} |u_{\varepsilon}| _{3}
 |u_{\varepsilon} - u_{0}|_{3} = o_{\varepsilon}(1)$$
and as in Lemma \ref{convergenciafeFcritical}, simply using the fact that $\|u_{\varepsilon}\|\leq 2\pi/(\alpha_{0}+1)$,
$$\left|\int_{\Omega} f(u_{\varepsilon}) u_{\varepsilon} - \int_{\Omega}f(u_{\varepsilon}) u_{0} \right| = o_{\varepsilon}(1).$$
Then from \eqref{eq:eps} we deduce  $\|u_{\varepsilon}\|\to \|u_{0}\|$ and so
 \begin{equation}\label{eq:forte}
 \lim_{\varepsilon\to0^{+}}u_{\varepsilon}= u_{0} \ \text{ in } \ \ H_{0}^{1}(\Omega).
 \end{equation}

Moreover
 $$
0= J_{\varepsilon}'(u_\varepsilon)[u_\varepsilon] \geq\|u_{\varepsilon}\|^{2}  -\int_{\Omega} f(u_{\varepsilon})u_{\varepsilon}\geq \|u_{\varepsilon}\|^{2} - \tau \int_{\Omega}|u_{\varepsilon}|^{r}\geq
\|u_{\varepsilon}\|^{2} - C \|u_{\varepsilon}\|^{r},
 $$
which implies (being $r>2$) that there exists a constant $h>0$ such that,
\begin{equation*}\label{eq:limiteh}
\forall \varepsilon>0 : \ \ 0<h\leq \|u_{\varepsilon}\|.
\end{equation*}
In particular $u_{0}\neq0$ and then also $\phi_{0}(u_{0})\neq0$; moreover from 
Lemma \ref{estimatetominimizingsequence},
$\|u_{\varepsilon}\|< \overline T(\varepsilon) $
and then  $\overline  T(\varepsilon)\not\to0$ as $\varepsilon\to 0^{+}$.

%
%
%

  
Finally, from $J_{\varepsilon}'(u_{\varepsilon})= 0$ we get
\begin{equation}\label{eq:epsilon}
\forall v\in H^{1}_{0}(\Omega): \int_{\Omega} \nabla u_{\varepsilon} \nabla v 
+\int_{\Omega}\phi_{\varepsilon}(u_{\varepsilon}) u_{\varepsilon}v
=\int_{\Omega} f(u_{\varepsilon})v.
\end{equation}
We want to pass to the limit in $\varepsilon$ in the identity above. 
Since, by Lemma \ref{lem:kavianlem}, $\phi_{\varepsilon}(u_{\varepsilon})\to \phi_{0}(u_{0})$ in $L^{3}(\Omega)$,
$u_{\varepsilon}\to u_{0}$ in $L^{3}(\Omega)$ and $v\in L^{3}(\Omega)$, by the H\"older inequality
we have
\begin{equation}\label{eq:prima}
\lim_{\varepsilon\to0^{+}}\int_{\Omega} \phi_{\varepsilon}(u_{\varepsilon})u_{\varepsilon}v = 
\int_{\Omega} \phi_{0}(u_{0})u_{0}v.
\end{equation}
On the other hand, again as in Lemma \ref{convergenciafeFcritical},
\begin{equation}\label{eq:seconda}
\lim_{\varepsilon\to0^{+}}\int_{\Omega} f(u_{\varepsilon})v =\int_{\Omega}f(u_{0})v.
\end{equation}

Then by \eqref{eq:epsilon}-\eqref{eq:seconda}, we get
$$\forall v\in H^{1}_{0}(\Omega): \int_{\Omega} \nabla u_{0} \nabla v 
+\int_{\Omega}\phi_{0}(u_{0}) u_{0}v
=\int_{\Omega} f(u_{0})v,$$
showing that the pair $u_{0},\phi_{0}:=\phi_{0}(u_{0})$  solves the Schr\"odinger-Poisson system.
Then in view of the first limit in Lemma \ref{lem:kavianlem} and \eqref{eq:forte},
 Theorem \ref{th:2} is completely proved.


\end{document}